\documentclass[reqno]{amsart}
\usepackage{amsmath,amsthm,amssymb,mathrsfs,stmaryrd,eucal,mathbbol}
\usepackage[all,cmtip]{xy}
\usepackage{hyperref}

\numberwithin{equation}{section}

\newcommand\void[1]{}

\newcommand{\C}{\mathbb{C}}

\newcommand{\CB}{\mathcal{B}}
\newcommand{\CC}{\mathcal{C}}
\newcommand{\CD}{\mathcal{D}}
\newcommand{\CE}{\mathcal{E}}

 \DeclareMathOperator{\Id}{Id}
 \DeclareMathOperator{\id}{id}

\newcommand{\adj}[4]{\xymatrix{ #1 \ar@<.5ex>[r]^-{#3} & #2 \ar@<.5ex>[l]^-{#4}}}

\newcommand{\Cat}{\mathcal{C}at}

\newcommand{\rev}{\mathrm{rev}}
\newcommand{\op}{\mathrm{op}}

\newcommand{\one}{\mathbf1}

\newtheorem{thm}{Theorem}[section]

\newtheorem{prop}[thm]{Proposition}
\newtheorem{cor}[thm]{Corollary}

\newtheorem{prop-defn}[thm]{Proposition-Definition}

\theoremstyle{definition}
\newtheorem{defn}[thm]{Definition}
\newtheorem{const}[thm]{Construction}
\newtheorem{exam}[thm]{Example}
\newtheorem{rem}[thm]{Remark}

\theoremstyle{remark}

\begin{document}

\begin{abstract}
We define the Drinfeld center of a monoidal category enriched over a braided monoidal category, and show that every modular tensor category can be realized in a canonical way as the Drinfeld center of a self-enriched monoidal category. We also give a generalization of this result for important applications in physics.
\end{abstract}

\title{Drinfeld center of enriched monoidal categories}
\author{Liang Kong}
\address{Yau Mathematical Sciences center, Tsinghua University, Beijing 100084, China (current address) \\ Department of Mathematics and Statistics, University of New Hampshire, Durham, NH 03824, USA}
\email{lkong@math.tsinghua.edu.cn}
\author{Hao Zheng}
\address{Department of Mathematics, Peking University, Beijing 100871, China}
\email{hzheng@math.pku.edu.cn}
\maketitle

\section{Introduction}

Enriched categories have been extensively studied in the past decades since they were introduced in \cite{EK}. Monoidal categories enriched over {\em symmetric} monoidal categories were also used implicitly or explicitly in the study of many categorical problems. For example, linear monoidal categories are enriched over the symmetric monoidal category of vector spaces. However, monoidal categories enriched over {\em braided} monoidal categories are almost vacant in the literature. It was not until recently that a definition was written down in \cite{bm,MP}. This delay is partly because categories enriched over braided monoidal categories behave poorly under Cartesian product \cite{js2}, which is one of the fundamental constructions in category theory.

In fact, the notion of a monoidal category enriched over a braided monoidal category is not as poor as it first looks. We show that not only one is able to generalize the Drinfeld center of a monoidal category to that of an enriched monoidal category (see Definition \ref{def:center}), but also this generalization shares many nice properties with the ordinary one, for example, it is an enriched braided monoidal category (see Theorem \ref{thm:center} and Definition \ref{def:braid}).

More importantly, this notion leads to a positive answer to the following question: given a modular tensor category $\CC$, is there any mathematical object whose ``center'' is $\CC$? This question is crucial to the study of 2+1D TQFT such as Chern-Simons theory, Reshetikhin-Turaev extended TQFT \cite{He1,He2,zheng} and topological orders with gapless edges \cite{kz}. We show that a modular tensor category (more generally, a nondegenerate braided fusion category) $\CC$ can be realized in a canonical way as the Drinfeld center of a self-enriched monoidal category (see Corollary \ref{cor:mtc}). We also give a generalization of this result in Corollary \ref{cor:fusion}, which has an important application in physics \cite{kz}.


\medskip

\noindent {\bf Acknowledgement.} HZ is supported by NSFC under Grant No. 11131008.

\section{Enriched (monoidal) categories}

First, we recall the notion of a (monoidal) category enriched over a (braided) monoidal category. See \cite{Ke,MP} and references therein.
Let $\CB$ be a monoidal category with tensor unit $\one$ and tensor product $\otimes:\CB\times\CB\to\CB$. We denote the identity morphism by $1: x \to x$ for all $x\in \CB$. The notation $\id_x$ is reserved for something else.

\smallskip

A {\em category $\CC^\sharp$ enriched over $\CB$} consists of a set of objects $Ob(\CC^\sharp)$, a hom object $\CC^\sharp(x,y)\in\CB$ for every pair $x,y\in\CC^\sharp$, a morphism (the identity morphism) $\id_x:\one\to\CC^\sharp(x,x)$ for every $x\in\CC^\sharp$ and a morphism (the composition law) $\circ:\CC^\sharp(y,z)\otimes\CC^\sharp(x,y)\to\CC^\sharp(x,z)$ for every triple $x,y,z\in\CC^\sharp$ rendering the following diagrams commutative for $x,y,z,w\in\CC^\sharp$:
$$\xymatrix@R=1.5em@!C=15ex{
  & \CC^\sharp(x,y)\otimes\CC^\sharp(x,x) \ar[rd]^\circ \\
  \CC^\sharp(x,y) \ar[rr]^1 \ar[ru]^{1\otimes\id_x} && \CC^\sharp(x,y), \\
}
$$
$$\xymatrix@R=1.5em@!C=15ex{
  & \CC^\sharp(y,y)\otimes\CC^\sharp(x,y) \ar[rd]^\circ \\
  \CC^\sharp(x,y) \ar[rr]^1 \ar[ru]^{\id_y\otimes 1} && \CC^\sharp(x,y), \\
}
$$
$$\xymatrix{
  \CC^\sharp(z,w)\otimes\CC^\sharp(y,z)\otimes\CC^\sharp(x,y) \ar[r]^-{1\otimes\circ} \ar[d]_{\circ\otimes 1} & \CC^\sharp(z,w)\otimes\CC^\sharp(x,z) \ar[d]^\circ \\
  \CC^\sharp(y,w)\otimes\CC^\sharp(x,y) \ar[r]^-\circ & \CC^\sharp(x,w). \\
}
$$
The {\em underlying category} of $\CC^\sharp$ is a category $\CC$ which has the same objects as $\CC^\sharp$ and has the sets of morphisms defined by $\CC(x,y) = \CB(\one,\CC^\sharp(x,y))$. A morphism in $\CC$ is also referred to as a morphism in $\CC^\sharp$.


\smallskip

An {\em enriched functor} $F:\CC^\sharp\to\CD^\sharp$ between two categories enriched over $\CB$ consists of a map $F:Ob(\CC^\sharp)\to Ob(\CD^\sharp)$ and a morphism $F: \CC^\sharp(x,y)\to\CD^\sharp(F(x),F(y))$ for every pair $x,y\in\CC^\sharp$ such that the following diagrams commute for $x,y,z\in\CC^\sharp$:
$$
\xymatrix@R=1em{
  & \one \ar[ld]_{\id_x} \ar[rd]^{\id_{F(x)}} \\
  \CC^\sharp(x,x) \ar[rr]^-F && \CD^\sharp(F(x),F(x)), \\
}
$$
$$\xymatrix{
  \CC^\sharp(y,z)\otimes\CC^\sharp(x,y) \ar[r]^-\circ \ar[d]_{F\otimes F} & \CC^\sharp(x,z) \ar[d]^F \\
 \CD^\sharp(F(y),F(z))\otimes\CD^\sharp(F(x),F(y)) \ar[r]^-\circ & \CD^\sharp(F(x),F(z)). \\
}
$$
An enriched functor $F:\CC^\sharp\to\CD^\sharp$ induces a functor between the underlying categories $F:\CC\to\CD$.

\smallskip

An {\em enriched natural transformation} $\xi:F\to G$ between two enriched functors $F,G:\CC^\sharp\to\CD^\sharp$ is a natural transformation between the underlying functors such that the following diagram commutes for $x,y\in\CC^\sharp$:
$$
\xymatrix@!C=25ex{
  \CC^\sharp(x,y) \ar[r]^-{G} \ar[d]_{F} & \CD^\sharp(G(x),G(y)) \ar[d]^{-\circ\xi_x} \\
  \CD^\sharp(F(x),F(y)) \ar[r]^-{\xi_y\circ-} & \CD^\sharp(F(x),G(y)). \\
}
$$

\begin{exam} \label{exam:B1}
Suppose $\CB$ is rigid. Then $\CB$ is canonically promoted to a self-enriched category $\CB^\sharp$: $Ob(\CB^\sharp)=Ob(\CB)$,
$\CB^\sharp(x,y) = y\otimes x^*$ where $x^*$ is the left dual of $x$,
the composition law $\circ:(z\otimes y^*)\otimes(y\otimes x^*)\to z\otimes x^*$ is induced by the counit map $v_y:y^*\otimes y\to\one$, and $\id_x$ is given by the unit map $u_x: \one\to x\otimes x^*$.
\end{exam}


Now we assume $\CB$ is a braided monoidal category with braiding $c_{x,y}:x\otimes y\to y\otimes x$.

\smallskip

Let $\CC^\sharp,\CD^\sharp$ be categories enriched over $\CB$. The {\em Cartesian product} $\CC^\sharp\times\CD^\sharp$  is a category enriched over $\CB$ defined as follows:
\begin{itemize}
  \item $Ob(\CC^\sharp\times\CD^\sharp)=Ob(\CC^\sharp)\times Ob(\CD^\sharp)$;
  \item $(\CC^\sharp\times\CD^\sharp)((x,y),(x',y')) = \CC^\sharp(x,x')\otimes\CD^\sharp(y,y')$;
  \item $\id_{(x,y)}:\one\to(\CC^\sharp\times\CD^\sharp)((x,y),(x,y))$ is given by $\one\simeq\one\otimes\one\xrightarrow{\id_x\otimes\id_y}\CC^\sharp(x,x)\otimes\CD^\sharp(y,y)$;
  \item the composition law
\begin{eqnarray*}
& \circ: (\CC^\sharp\times\CD^\sharp)((x',y'),(x'',y'')) \otimes (\CC^\sharp\times\CD^\sharp)((x,y),(x',y')) \\
& \to (\CC^\sharp\times\CD^\sharp)((x,y),(x'',y''))
\end{eqnarray*}
is given by
\begin{eqnarray*}
& & \CC^\sharp(x',x'')\otimes\CD^\sharp(y',y'')\otimes\CC^\sharp(x,x')\otimes\CD^\sharp(y,y') \\
& \xrightarrow{1\otimes c^{-1}\otimes 1} & \CC^\sharp(x',x'')\otimes\CC^\sharp(x,x')\otimes\CD^\sharp(y',y'')\otimes\CD^\sharp(y,y') \\
& \xrightarrow{~~\circ\otimes\circ~~} & \CC^\sharp(x,x'')\otimes\CD^\sharp(y,y'').
\end{eqnarray*}
\end{itemize}
If $F:\CC^\sharp\to\CD^\sharp$ and $F':\CC'^\sharp\to\CD'^\sharp$ are enriched functors, there is an obvious enriched functor $F\times F':\CC^\sharp\times\CC'^\sharp\to\CD^\sharp\times\CD'^\sharp$.

\begin{rem}
The underlying category of $\CC^\sharp\times\CD^\sharp$ has the same objects as $\CC\times\CD$, but may have quite different morphisms. Nevertheless, there is an obvious functor from $\CC\times\CD$ to the underlying category of $\CC^\sharp\times\CD^\sharp$, therefore an enriched functor $F:\CC^\sharp\times\CD^\sharp\to\CE^\sharp$ induces a functor $F:\CC\times\CD\to\CE$.
\end{rem}

We have a canonical equivalence $(\CC^\sharp\times\CD^\sharp)\times\CE^\sharp \simeq \CC^\sharp\times(\CD^\sharp\times\CE^\sharp)$ for enriched categories $\CC^\sharp,\CD^\sharp,\CE^\sharp$. However, $\CC^\sharp\times\CD^\sharp \not\simeq \CD^\sharp\times\CC^\sharp$ in general unless $\CB$ is symmetric. So, the categories enriched over $\CB$ together with the enriched functors and enriched natural isomorphisms form a monoidal 2-category $\Cat^\CB$ (in which 2-morphisms are all invertible). Then we have the notions of associative algebras, modules and duality in $\Cat^\CB$ (see, for example, \cite{Lu}), but there is no obvious notion of commutativity. Unwinding the definition, we see that an associative algebra in $\Cat^\CB$ is nothing but an enriched monoidal category defined below. (Similarly, one defines enriched module categories, tensor product of enriched module categories, etc.)

\begin{defn} \label{def:emc}
A {\em monoidal category enriched over $\CB$} consists of a category $\CC^\sharp$ enriched over $\CB$, an object $\one_{\CC^\sharp}\in\CC^\sharp$, an enriched functor $\otimes:\CC^\sharp\times\CC^\sharp\to \CC^\sharp$, and enriched natural isomorphisms $\lambda:\one_{\CC^\sharp}\otimes- \to \Id_{\CC^\sharp}$, $\rho: -\otimes\one_{\CC^\sharp} \to \Id_{\CC^\sharp}$, $\alpha: (-\otimes-)\otimes- \to -\otimes(-\otimes-)$ such that the underlying category $\CC$, the object $\one_{\CC^\sharp}$, the underlying functor $\otimes:\CC\times\CC\to\CC$ and the natural isomorphisms $\lambda,\rho,\alpha$ define an ordinary monoidal category (in another word, $\lambda,\rho,\alpha$ satisfy the triangle axiom and the pentagon axiom of a monoidal category).

An {\em enriched monoidal functor} between two enriched monoidal categories $\CC^\sharp,\CD^\sharp$ consists of an enriched functor $F:\CC^\sharp\to\CD^\sharp$, an isomorphism $\mu:F(\one_{\CC^\sharp})\to\one_{\CD^\sharp}$ and an enriched natural isomorphism $\nu:F(-)\otimes F(-) \to F(-\otimes-)$ such that the underlying functor $F:\CC\to\CD$ together with $\mu,\nu$ defines an ordinary monoidal functor.

An {\em enriched monoidal natural transformation} $\xi:F\to G$ between two enriched monoidal functors is an enriched natural transformation which also defines an ordinary monoidal natural transformation between the underlying monoidal functors.
\end{defn}

A notion of a monoidal category enriched over a duoidal category was introduced in \cite{bm}. An enriched monoidal category is {\em strict} if $\lambda,\rho,\alpha$ are identity; this recovers \cite[Definition 2.1]{MP}. The MacLane strictness theorem is also true in the enriched setting:


\begin{prop}
Let $\CC^\sharp$ be an enriched monoidal category. There is an enriched monoidal equivalence $\CC^\sharp\simeq\CD^\sharp$ where $\CD^\sharp$ is a strict enriched monoidal category.
\end{prop}

\begin{proof}
According to the MacLane strictness theorem, there is a monoidal equivalence $F:\CD\to\CC$ where $\CD$ is a strict monoidal category. To lift $\CD$ to a strict enriched monoidal category $\CD^\sharp$, we set $\CD^\sharp(x,y) = \CC^\sharp(F(x),F(y))$ and let the composition law and the enriched structure of the tensor product be induced from those of $\CC^\sharp$. Then $F$ automatically lifts to an enriched monoidal equivalence.
\end{proof}

\begin{exam} \label{exam:B2}
Suppose $\CB$ is rigid. Then $\CB$ can be canonically promoted to a monoidal category $\CB^\sharp$ enriched over $\CB$ \cite[Section 2.3]{MP}. In fact, one needs to promote $\otimes:\CB\times\CB\to\CB$ to a well-defined enriched functor. It turns out that one should take $\otimes:(\CB^\sharp\times\CB^\sharp)((x,y),(x',y')) \to \CB^\sharp(x\otimes y,x'\otimes y')$ to be $1 \otimes c_{x^*,y'\otimes y^*}: (x'\otimes x^*)\otimes(y'\otimes y^*) \to (x'\otimes y')\otimes(x\otimes y)^*$.
\end{exam}

\section{Enriched braided/symmetric monoidal categories}

Keep the assumption that $\CB$ is a braided monoidal category with braiding $c_{x,y}:x\otimes y\to y\otimes x$.

\begin{prop} \label{prop:comm-equiv}
Let $\CC^\sharp,\CD^\sharp,\CE^\sharp$ be categories enriched over $\CB$. The following conditions are equivalent for an enriched functor $F:\CC^\sharp\times\CD^\sharp\to\CE^\sharp$:
\begin{enumerate}
\item The assignment $F^\rev(y,x)=F(x,y)$ and the composite morphism
\begin{align*}
F^\rev:~ & (\CD^\sharp\times\CC^\sharp)((y,x),(y',x')) = \CD^\sharp(y,y')\otimes\CC^\sharp(x,x') \\
& \xrightarrow{~c~} \CC^\sharp(x,x')\otimes\CD^\sharp(y,y') \xrightarrow{F} \CE^\sharp(F^\rev(y,x),F^\rev(y',x'))
\end{align*}
define an enriched functor $F^\rev:\CD^\sharp\times\CC^\sharp\to\CE^\sharp$.
\item The following diagram commutes for $x,x'\in\CC^\sharp$ and $y,y'\in\CD^\sharp$:
$$
\xymatrix@!C=15ex{
  & \CC^\sharp(x,x')\otimes\CD^\sharp(y,y') \ar[rd]^F \\
  \CC^\sharp(x,x')\otimes\CD^\sharp(y,y') \ar[rr]^F \ar[ru]^{c^2} && \CE^\sharp(F(x,y),F(x',y')). \\
}
$$
\end{enumerate}
\end{prop}

\begin{proof}
It is clear that the following diagram commutes:
$$
\xymatrix{
  & \one \ar[ld]_{\id_{(y,x)}} \ar[rd]^{\id_{F^\rev(y,x)}} \\
  (\CD^\sharp\times\CC^\sharp)((y,x),(y,x)) \ar[rr]^-{F^\rev} && \CE^\sharp(F^\rev(y,x),F^\rev(y,x)). \\
}
$$
So, Condition $(1)$ is equivalent to the commutativity of the following diagram:
$$\scriptsize
\xymatrix@!C=55ex{
  (\CD^\sharp\times\CC^\sharp)((y',x'),(y'',x''))\otimes(\CD^\sharp\times\CC^\sharp)((y,x),(y',x')) \ar[r]^-\circ \ar[d]_{F^\rev\otimes F^\rev} & (\CD^\sharp\times\CC^\sharp)((y,x),(y'',x'')) \ar[d]^{F^\rev} \\
 \CE^\sharp(F^\rev(y',x'),F^\rev(y'',x''))\otimes\CE^\sharp(F^\rev(y,x),F^\rev(y',x')) \ar[r]^-\circ & \CE^\sharp(F^\rev(y,x),F^\rev(y'',x'')). \\
}
$$
This amounts to say that the following two composite morphisms coincide
\begin{eqnarray}
&& \CC^\sharp(x',x'')\otimes\CD^\sharp(y',y'')\otimes\CC^\sharp(x,x')\otimes\CD^\sharp(y,y') \nonumber \\
& \xrightarrow{1\otimes c^{\pm1}\otimes 1} & \CC^\sharp(x',x'')\otimes\CC^\sharp(x,x')\otimes\CD^\sharp(y',y'')\otimes\CD^\sharp(y,y') \nonumber \\
& \xrightarrow{\circ\otimes\circ} & \CC^\sharp(x,x'')\otimes\CD^\sharp(y,y'') \label{eqn:comm1} \\
& \xrightarrow{~~F~~} & \CE^\sharp(F(x,y),F(x'',y'')). \nonumber
\end{eqnarray}
On the other hand side, Condition $(2)$ is equivalent to that the following two composite morphisms coincide
\begin{equation*}
\CD^\sharp(y',y'')\otimes\CC^\sharp(x,x') \xrightarrow{c^{\pm1}} \CC^\sharp(x,x')\otimes\CD^\sharp(y',y'') \xrightarrow{~~F~~} \CE^\sharp(F(x,y'),F(x',y'')). 
\end{equation*}
Setting $x'=x''$ and $y=y'$ in \eqref{eqn:comm1}, we see that $(1)\Rightarrow(2)$. Moreover, \eqref{eqn:comm1} is identical to
{\scriptsize
\begin{eqnarray*}
&& \CC^\sharp(x',x'')\otimes\CD^\sharp(y',y'')\otimes\CC^\sharp(x,x')\otimes\CD^\sharp(y,y') \\
& \xrightarrow{1\otimes\id_{y''}\otimes c^{\pm1}\otimes\id_x\otimes 1} & \CC^\sharp(x',x'')\otimes\CD^\sharp(y'',y'')\otimes\CC^\sharp(x,x')\otimes\CD^\sharp(y',y'')\otimes\CC^\sharp(x,x)\otimes\CD^\sharp(y,y') \\
& \xrightarrow{F\otimes F\otimes F} & \CE^\sharp(F(x',y''),F(x'',y''))\otimes\CE^\sharp(F(x,y'),F(x',y''))\otimes\CE^\sharp(F(x,y),F(x,y')) \\
& \xrightarrow{~~\circ~~} & \CE^\sharp(F(x,y),F(x'',y'')).
\end{eqnarray*}
}
Therefore, $(2)\Rightarrow(1)$.
\end{proof}

\begin{defn} \label{def:comm}
We say that an enriched functor $F:\CC^\sharp\times\CD^\sharp\to\CE^\sharp$ is {\em commutative} if it satisfies the equivalent conditions from Proposition \ref{prop:comm-equiv}.
\end{defn}

\begin{rem}
(1) By definition, if $F:\CC^\sharp\times\CD^\sharp\to\CE^\sharp$ is a commutative enriched functor, $F^\rev$ is also commutative and we have $(F^\rev)^\rev = F$.

(2) If $\xi:F\to G$ is an enriched natural transformation between commutative enriched functors $F,G:\CC^\sharp\times\CD^\sharp\to\CE^\sharp$, then $\xi$ defines an enriched natural transformation $\xi^\rev:F^\rev\to G^\rev$.

(3) If $\CC^\sharp$ is an enriched monoidal category with a commutative tensor product $\otimes$, then $\otimes^\rev$ also provides an enriched monoidal structure on $\CC^\sharp$.

(4) If $\CB$ is symmetric, any enriched functor $F:\CC^\sharp\times\CD^\sharp\to\CE^\sharp$ is commutative.
\end{rem}

One way to define an enriched braided/symmetric monoidal category is as follows.

\begin{defn} \label{def:braid}
An {\em enriched braided monoidal category} consists of an enriched monoidal category $\CC^\sharp$ such that the tensor product $\otimes:\CC^\sharp\times\CC^\sharp\to\CC^\sharp$ is commutative, as well as an enriched natural isomorphism $\beta:\otimes\to\otimes^\rev$ such that $\beta$ defines a braiding for the underlying monoidal category $\CC$.

An {\em enriched braided monoidal functor} $F:\CC^\sharp\to\CD^\sharp$ between enriched braided monoidal categories $\CC^\sharp,\CD^\sharp$ is an enriched monoidal functor such that the underlying monoidal functor $F:\CC\to\CD$ is a braided monoidal functor.

An {\em enriched symmetric monoidal category} is an enriched braided monoidal category whose underlying braided monoidal category is symmetric (in another word, $\beta^\rev\circ\beta=\Id_\otimes$). An {\em enriched symmetric monoidal functor} between enriched symmetric monoidal categories is simply an enriched braided monoidal functor.
\end{defn}

\begin{exam} \label{exam:single-obj}
Let $A$ be a commutative algebra in $\CB$. We have a strict symmetric monoidal category $\CC^\sharp$ enriched over $\CB$: it has a single object $\one_{\CC^\sharp}$, $\CC^\sharp(\one_{\CC^\sharp},\one_{\CC^\sharp})=A$, the composition law $\circ:A\otimes A\to A$ is the multiplication of $A$, $\id_{\one_{\CC^\sharp}}:\one\to A$ is the unit of $A$, and the morphism $\otimes:A\otimes A\to A$ is also the multiplication of $A$. Conversely, any strict enriched monoidal category with a single object arises in this way. Therefore, the Cartesian product $\CC^\sharp\times\CC^\sharp$ does not admit an enriched monoidal structure unless the algebra $A\otimes A$ in $\CB$ is commutative.
\end{exam}

\section{Drinfeld center}

In what follows, we assume that $\CB$ is a braided monoidal category satisfying the following condition:
\begin{enumerate}
\renewcommand\theenumi{$*$}
\item\label{item:star}
$\CB$ admits equalizers and the intersection of arbitrary many subobjects of any object $x\in\CB$ exists.
\end{enumerate}
The Drinfeld center \cite{Ma,JS} of a monoidal category has a straightforward generalization:

\begin{defn}  \label{def:half-braiding}
Let $\CC^\sharp$ be a monoidal category enriched over $\CB$. A {\em half-braiding} for an object $x\in\CC^\sharp$ is an enriched natural isomorphism $b_x:x\otimes-\to-\otimes x$ between enriched endo-functors of $\CC^\sharp$ such that it defines a half-braiding in the underlying monoidal category $\CC$.
\end{defn}

\begin{defn} \label{def:center}
The {\em Drinfeld center} of $\CC^\sharp$ is a monoidal category $Z(\CC^\sharp)$ enriched over $\CB$ that is defined as follows:
\begin{itemize}
  \item an object is a pair $(x,b_x)$ where $x\in\CC^\sharp$ and $b_x$ is a half-braiding for $x$;
  \item $Z(\CC^\sharp)((x,b_x),(y,b_y))$ is the intersection of the equalizers of the diagrams $\CC^\sharp(x,y) \rightrightarrows \CC^\sharp(x\otimes z,z\otimes y)$ depicted below for all $z\in\CC^\sharp$
\begin{equation} \label{diag:equalizer}
\raisebox{2em}{\xymatrix@!C=25ex{
  \CC^\sharp(x,y) \ar[r]^-{\otimes\circ(\id_z\otimes 1)} \ar[d]_{\otimes\circ(1\otimes\id_z)} & \CC^\sharp(z\otimes x,z\otimes y) \ar[d]^{-\circ b_{x,z}} \\
  \CC^\sharp(x\otimes z,y\otimes z) \ar[r]^-{b_{y,z}\circ-} & \CC^\sharp(x\otimes z,z\otimes y); \\
}}
\end{equation}
\item the identity morphisms, the composition law and the enriched monoidal structure are induced from those of $\CC^\sharp$ (see the proof of Proposition \ref{prop:center}).
\end{itemize}
\end{defn}

It is routine to prove the following proposition. For reader's convenience, we provide some details of the proof.

\begin{prop} \label{prop:center}
The Drinfeld center $Z(\CC^\sharp)$ of $\CC^\sharp$ is a well-defined enriched monoidal category. The underlying category of $Z(\CC^\sharp)$ is a full subcategory of $Z(\CC)$.
\end{prop}

\begin{proof}
The identity morphism $\id_x: \one \to \CC^\sharp(x,x)$ equalizes the diagram $\CC^\sharp(x,x) \rightrightarrows \CC^\sharp(x\otimes z,z\otimes x)$ for $(x,b_x)\in Z(\CC^\sharp)$ and $z\in \CC^\sharp$ because $b_x$ defines a half-braiding in the underlying monoidal category $\CC$. Therefore, $\id_x$ factors though $Z(\CC^\sharp)((x,b_x),(x,b_x))$.
Moreover, we have a commutative diagram:
$$\small\xymatrix{
  Z(\CC^\sharp)((w,b_w),(y,b_y))\otimes Z(\CC^\sharp)((x,b_x),(w,b_w)) \ar[r] \ar[d] & \CC^\sharp(z\otimes w,z\otimes y)\otimes\CC^\sharp(z\otimes x,z\otimes w) \ar[d]^{1\otimes(-\circ b_{x,z})} \\
  \CC^\sharp(w\otimes z,y\otimes z)\otimes\CC^\sharp(x\otimes z,w\otimes z) \ar[r]^-{(b_{y,z}\circ-\circ b_{w,z}^{-1})\otimes(b_{w,z}\circ-)} & \CC^\sharp(z\otimes w,z\otimes y)\otimes\CC^\sharp(x\otimes z,z\otimes w)
}
$$
which implies that the composite morphism
$$Z(\CC^\sharp)((w,b_w),(y,b_y))\otimes Z(\CC^\sharp)((x,b_x),(w,b_w)) \to \CC^\sharp(w,y)\otimes\CC^\sharp(x,w) \xrightarrow{\circ} \CC^\sharp(x,y)$$
equalizes $\CC^\sharp(x,y) \rightrightarrows \CC^\sharp(x\otimes z,z\otimes y)$ hence factors through $Z(\CC^\sharp)((x,b_x),(y,b_y))$. This shows that the identity morphisms and composition law of $\CC^\sharp$ induce those of $Z(\CC^\sharp)$, rendering $Z(\CC^\sharp)$ a well-defined enriched category.

We have a commutative diagram:
$$\xymatrix{
  \CC^\sharp(z\otimes x,z\otimes y)\otimes\CC^\sharp(x',y') \ar[r]^-{(b_{y,z}^{-1}\circ-\circ b_{x,z})\otimes1} & \CC^\sharp(x\otimes z,y\otimes z)\otimes\CC^\sharp(x',y') \ar[d]^\otimes \\
  Z(\CC^\sharp)((x,b_x),(y,b_y))\otimes Z(\CC^\sharp)((x',b_{x'}),(y',b_{y'})) \ar[u] \ar[d] \ar[ru] \ar[rd] & \CC^\sharp(x\otimes z\otimes x',y\otimes z\otimes y') \\
  \CC^\sharp(x,y)\otimes\CC^\sharp(x'\otimes z,y'\otimes z) \ar[r]^-{1\otimes(b_{y',z}\circ-\circ b_{x',z}^{-1})} & \CC^\sharp(x,y)\otimes\CC^\sharp(z\otimes x',z\otimes y') \ar[u]_\otimes
}
$$
the outer square of which implies that the composite morphism {\small
$$Z(\CC^\sharp)((x,b_x),(y,b_y))\otimes Z(\CC^\sharp)((x',b_{x'}),(y',b_{y'})) \to \CC^\sharp(x,y)\otimes\CC^\sharp(x',y') \xrightarrow{\otimes} \CC^\sharp(x\otimes x', y\otimes y')
$$
}equalizes $\CC^\sharp(x\otimes x', y\otimes y') \rightrightarrows \CC^\sharp(x\otimes x'\otimes z,z\otimes y\otimes y')$ hence factors through $Z(\CC^\sharp)((x\otimes x',b_{x\otimes x'}), (y\otimes y',b_{y\otimes y'}))$. This shows that the enriched functor $\otimes:\CC^\sharp\times\CC^\sharp\to\CC^\sharp$ induces a well-defined enriched functor $\otimes:Z(\CC^\sharp)\times Z(\CC^\sharp)\to Z(\CC^\sharp)$.

Note that $\CB(\one, Z(\CC^\sharp)((x,b_x),(y,b_y)))$ consists of those morphisms in $\CB(\one, \CC^\sharp(x,y))$ equalizing Diagram \eqref{diag:equalizer} for all $z\in\CC^\sharp$, i.e. those morphisms in $\CC(x,y)$ intertwining the half-braidings $b_x$ and $b_y$, i.e. those morphisms in $Z(\CC)((x,b_x),(y,b_y))$. This shows that the underlying category of $Z(\CC^\sharp)$ is a full subcategory of $Z(\CC)$.
As a consequence, the enriched natural isomorphisms $\lambda,\rho,\alpha$ of $\CC^\sharp$ induce those of $Z(\CC^\sharp)$ rendering $Z(\CC^\sharp)$ a well-defined enriched monoidal category.
\end{proof}

\begin{thm} \label{thm:center}
Let $\CC^\sharp$ be a monoidal category enriched over $\CB$.
\begin{enumerate}
\item The composite enriched functors $L: Z(\CC^\sharp)\times\CC^\sharp \to \CC^\sharp\times\CC^\sharp \xrightarrow\otimes \CC^\sharp$ and $R: \CC^\sharp\times Z(\CC^\sharp) \to \CC^\sharp\times\CC^\sharp \xrightarrow\otimes \CC^\sharp$ are commutative.

\item The natural isomorphism $b_{x,y}:x\otimes y\to y\otimes x$ for $x\in Z(\CC^\sharp)$ and $y\in\CC^\sharp$ defines an enriched natural isomorphism $L\simeq R^\rev$.

\item The Drinfeld center $Z(\CC^\sharp)$ is an enriched braided monoidal category in the sense of Definition \ref{def:braid} with the braiding $b_{x,y}:x\otimes y\to y\otimes x$.
\end{enumerate}
\end{thm}

\begin{proof}
$(1)$ We have a commutative diagram for $x,x'\in Z(\CC^\sharp)$ and $y,y'\in\CC^\sharp$:
$$
\xymatrix@!C=35ex{
  Z(\CC^\sharp)(x,x')\otimes\CC^\sharp(y,y') \ar[r]^-{c^{-1}} \ar[d]_{(R\circ(\id_{y'}\otimes 1))\otimes(R\circ(1\otimes\id_x))} \ar@/^.8pc/[rdd]^L \ar@/^1pc/[dd]^>>>>>>{(L\circ(1\otimes\id_{y'}))\otimes(L\circ(\id_x\otimes 1))} & \CC^\sharp(y,y')\otimes Z(\CC^\sharp)(x,x') \ar[d]^R \\
  \CC^\sharp(y'\otimes x,y'\otimes x')\otimes\CC^\sharp(y\otimes x,y'\otimes x) \ar[r]^-\circ \ar[d]_{(b_{x',y'}^{-1}\circ-\circ b_{x,y'})\otimes(b_{x,y'}^{-1}\circ-\circ b_{x,y})} & \CC^\sharp(y\otimes x,y'\otimes x') \ar[d]^{b_{x',y'}^{-1}\circ-\circ b_{x,y}} \\
  \CC^\sharp(x\otimes y',x'\otimes y')\otimes\CC^\sharp(x\otimes y,x\otimes y') \ar[r]^-\circ & \CC^\sharp(x\otimes y,x'\otimes y'), \\
}
$$
where the left triangle is due to the definition of Drinfeld center, the middle triangle and the two squares are clear from definition.
Whence we obtain a commutative diagram:
\begin{equation} \label{eqn:lr1}
\raisebox{2em}{\xymatrix{
  Z(\CC^\sharp)(x,x')\otimes\CC^\sharp(y,y') \ar[r]^-{c^{-1}} \ar[d]_L & \CC^\sharp(y,y')\otimes Z(\CC^\sharp)(x,x') \ar[d]^R \\
  \CC^\sharp(x\otimes y,x'\otimes y') & \ar[l]_-{b_{x',y'}^{-1}\circ-\circ b_{x,y}} \CC^\sharp(y\otimes x,y'\otimes x'). \\
}}
\end{equation}
Similarly, we have a commutative diagram:
\begin{equation} \label{eqn:lr2}
\raisebox{2em}{\xymatrix{
  Z(\CC^\sharp)(x,x')\otimes\CC^\sharp(y,y') \ar[d]_L & \CC^\sharp(y,y')\otimes Z(\CC^\sharp)(x,x') \ar[d]^R \ar[l]_-{c^{-1}} \\
  \CC^\sharp(x\otimes y,x'\otimes y') \ar[r]^-{b_{x',y'}\circ-\circ b_{x,y}^{-1}} & \CC^\sharp(y\otimes x,y'\otimes x'). \\
}}
\end{equation}
Comparing \eqref{eqn:lr1} with \eqref{eqn:lr2}, we see that $L$ and $R$ are commutative.

$(2)$ is clear from the commutative diagram \eqref{eqn:lr2}.

$(3)$ is a consequence of $(1)$ and $(2)$ because the enriched monoidal structure of $Z(\CC^\sharp)$ is induced from that of $\CC^\sharp$.
\end{proof}

\begin{exam}
Let $\CC^\sharp$ be an enriched monoidal category with a single object as shown in Example \ref {exam:single-obj}. Then $Z(\CC^\sharp)$ contains $\CC^\sharp$ as a full subcategory. This example shows that the Drinfeld center of a monoidal category enriched over $\CB$ is not necessarily enriched over a symmetric monoidal subcategory of $\CB$.
\end{exam}

\begin{rem}
It is possible to define the Drinfeld center $Z(\CC^\sharp)$ alternatively as the enriched category of $\CC^\sharp$-$\CC^\sharp$-bimodule functors $\CC^\sharp\to\CC^\sharp$ as in the unenriched case.
\end{rem}

An object $x$ of a braided monoidal category $\CB$ is {\em transparent} if the double braiding of $x$ with any object $y\in\CB$ is trivial: $c_{y,x}\circ c_{x,y}=1$.

\begin{thm} \label{thm:self-enrich}
Let $\CB$ be a rigid braided monoidal category satisfying Condition \eqref{item:star}, and let $\CB^\sharp$ be the enriched monoidal category constructed in Example \ref{exam:B1} and \ref{exam:B2}. The Drinfeld center $Z(\CB^\sharp)$ has the same objects as $\CB^\sharp$ and $Z(\CB^\sharp)(x,y)$ is the maximal transparent subobject of $\CB^\sharp(x,y)$.
\end{thm}

\begin{proof}
Let $b_x$ be a half-braiding for an object $x\in\CB^\sharp$. Since $b_x$ is an enriched natural isomorphism, we have by definition the following commutative diagram for $y,z\in\CB^\sharp$:
$$
\xymatrix@!C=25ex{
  \CB^\sharp(y,z) \ar[r]^-{\otimes\circ(1\otimes\id_x)} \ar[d]_{\otimes\circ(\id_x\otimes 1)} & \CB^\sharp(y\otimes x,z\otimes x) \ar[d]^{-\circ b_{x,y}} \\
  \CB^\sharp(x\otimes y,x\otimes z) \ar[r]^-{b_{x,z}\circ-} & \CB^\sharp(x\otimes y,z\otimes x).  \\
}
$$
Using the explicit construction of $\CB^\sharp$ in Example \ref{exam:B1} and \ref{exam:B2}, this commutative diagram is nothing but the following one:
\begin{equation} \label{diag:en-nat-iso}
\raisebox{2em}{\xymatrix@!C=30ex{
  z\otimes y^* \ar[r]^-{(1\otimes c_{y^*,x\otimes x^*})\circ(1\otimes u_x)} \ar[d]_{(1\otimes c_{x^*,z\otimes y^*})\circ(u_x\otimes 1)} & (z\otimes x)\otimes(y\otimes x)^* \ar[d]^{1\otimes b_{x,y}^*} \\
  (x\otimes z)\otimes(x\otimes y)^* \ar[r]^-{b_{x,z}\otimes 1^*} & (z\otimes x)\otimes(x\otimes y)^*. \\
}}
\end{equation}
When $y=\one$, we obtain the following special case:
$$
\xymatrix@!C=20ex{
  z \ar[r]^-{1\otimes u_x} \ar[d]_{(1\otimes c_{x^*,z})\circ(u_x\otimes 1)} & z\otimes x\otimes x^* \ar[d]^1 \\
  x\otimes z\otimes x^* \ar[r]^-{b_{x,z}\otimes 1} & z\otimes x\otimes x^*, \\
}
$$
which implies that $b_{x,z}=c_{x,z}$. Conversely, setting $b_x=c_{x,-}$ automatically renders Diagram \eqref{diag:en-nat-iso} commutative, hence defines an enriched natural isomorphism.
Therefore, $Z(\CB^\sharp)$ has the same objects as $\CB^\sharp$.

By definition, $Z(\CB^\sharp)((x,b_x),(y,b_y))$ is the maximal subobject $\iota:t\hookrightarrow y\otimes x^*$ rendering the following diagram commutative for all $z\in\CB^\sharp$:
$$
\xymatrix@!C=24ex{
  t \ar[r]^-{u_z\otimes\iota} \ar[d]_{\iota\otimes u_z} & (z\otimes z^*)\otimes(y\otimes x^*) \ar[r]^-{1\otimes c_{z^*,y\otimes x^*}} & (z\otimes y)\otimes(z\otimes x)^* \ar[d]^{1\otimes b_{x,z}^*} \\
  (y\otimes x^*)\otimes(z\otimes z^*) \ar[r]^-{1\otimes c_{x^*,z\otimes z^*}} & (y\otimes z)\otimes(x\otimes z)^* \ar[r]^-{b_{y,z}\otimes 1^*} & (z\otimes y)\otimes(x\otimes z)^*. \\
}
$$
It amounts to say that the double braiding of $t$ and $z^*$ is trivial for all $z\in\CB$. It follows that $t$ is the maximal transparent subobject of $y\otimes x^*$.
\end{proof}

\begin{rem}
The underlying category of $Z(\CB^\sharp)$ does not agree with $Z(\CB)$ in general. For example, when $\CB$ is a rigid symmetric monoidal category, we have $Z(\CB^\sharp)\simeq\CB^\sharp$ by Theorem \ref{thm:self-enrich}, whose underlying category is merely $\CB$.
\end{rem}

A {\em multi-fusion category} is a rigid semisimple $\C$-linear monoidal category with finitely many simple objects and finite-dimensional hom spaces (see, for example, \cite{EGNO}). It is called a {\em fusion category} if the tensor unit is simple. A fusion category $\CB$ is automatically enriched over the symmetric monoidal category of finite-dimensional vector spaces and the latter is canonically embedded in $\CB$ via the tensor unit of $\CB$. In this way (different from Example \ref{exam:B1}), $\CB$ can be viewed as a monoidal category enriched over itself. Similarly, a braided fusion category can be viewed as a braided monoidal category enriched over itself.
Clearly, a braided fusion category $\CB$ satisfies Condition \eqref{item:star}.

A braided fusion category is called {\em nondegenerate} if the tensor unit is the unique transparent simple object. For example, a modular tensor category (see, for example, \cite{BK,DGNO}) is a nondegenerate braided fusion category.

\begin{cor} \label{cor:mtc}
Let $\CB$ be a nondegenerate braided fusion category. Then we have $Z(\CB^\sharp)\simeq\CB$ as braided monoidal categories enriched over $\CB$.
\end{cor}
\begin{proof}
According to Theorem \ref{thm:self-enrich}, the objects in $Z(\CB^\sharp)$ coincide with those in $\CB$. Since $\one$ is the unique transparent simple object in $\CB$, Theorem \ref{thm:self-enrich} says that $Z(\CB^\sharp)(x,y)\simeq \CB(\one, \CB^\sharp(x,y)) \one = \CB(x,y) \one$.
\end{proof}

\begin{rem}
In recent works \cite{He1,He2}, certain unitary modular tensor categories (completed by separable Hilbert spaces) were shown to be the (usual) Drinfeld center of certain categories of solitons, and the latter were proposed as a candidate for the value of Chern-Simons theory on a point. We expect that a self-enriched modular tensor category $\CB^\sharp$ could be realized as the value of a fully extended Reshetikhin-Turaev TQFT on a point such that the value on a circle is $\CB$ (see \cite{zheng} for more details).
\end{rem}

\section{A generalization}

In this section, we give a generalization of Theorem \ref{thm:self-enrich} and Corollary \ref{cor:mtc}. This generalization is inspired by \cite{MP}. It was shown in \cite{MP} that, under some mild assumptions such as rigidity, all enriched monoidal categories arise from the following construction.

\smallskip

Let $\CB$ be a braided monoidal category satisfying Condition \eqref{item:star}. We use $\bar\CB$ to denote the same monoidal category $\CB$ but equipped with the anti-braiding $\bar c_{x,y} := c^{-1}_{y,x}$.

Let $\CC$ be a monoidal category equipped with a braided oplax monoidal functor $\psi:\bar\CB\to Z(\CC)$ such that $\psi$ induces an equivalence $\psi(\one)\simeq \one_{Z(\CC)}$. Then every object $w\in\CB$ is equipped with a half braiding $b_{\phi(w),-}$ in $\CC$, where $\phi: \bar\CB\to Z(\CC) \to \CC$ is the composition of $\psi$ with the forgetful functor.
Suppose $\CC$ has internal hom in $\CB$. That is, the functor $\phi(-)\otimes x:\CB\to\CC$ has a right adjoint $[x,-]:\CC\to\CB$ for every $x\in\CC$. In particular, we have a unit map $u_w:w\mapsto[x,\phi(w)\otimes x]$ for $w\in\CB$ and a counit map $v_y:\phi[x,y]\otimes x\to y$ for $y\in\CC$ associated to the adjunction. 

\begin{const} \label{const:enrich}
The monoidal category $\CC$ can be canonically promoted to a monoidal category $\CC^\sharp$ enriched over $\CB$. It has the same objects as $\CC$ and $\CC^\sharp(x,y)=[x,y]$. The composition law $\circ: [y,z]\otimes[x,y] \to [x,z]$ is induced by
$$\phi([y,z]\otimes[x,y])\otimes x \to \phi[y,z]\otimes\phi[x,y]\otimes x \xrightarrow{1 \otimes v_y} \phi[y,z]\otimes y \xrightarrow{v_z} z
$$
and $\id_x:\one\to[x,x]$ is induced by $\phi(\one)\otimes x \to x$.
To promote the tensor product $\otimes:\CC\times\CC\to\CC$ to a well-defined enriched functor, one should take
\begin{equation*}
\otimes:\CC^\sharp\times\CC^\sharp((x,y),(x',y')) \to \CC^\sharp(x\otimes y,x'\otimes y')
\end{equation*}
to be the morphism $[x,x']\otimes[y,y'] \to [x\otimes y,x'\otimes y']$ induced by
\begin{align*}
& \phi([x,x']\otimes[y,y'])\otimes x\otimes y \to \phi[x,x']\otimes\phi[y,y']\otimes x\otimes y \\
& \quad\quad\quad \xrightarrow{1 \otimes b_{\phi[y,y'],x}\otimes 1} \phi[x,x']\otimes x\otimes\phi[y,y']\otimes y \xrightarrow{v_{x'}\otimes v_{y'}} x'\otimes y'.
\end{align*}
We will refer to this construction of $\CC^\sharp$ as the {\em canonical construction} of enriched monoidal category from the pair $(\CB, \CC)$.
\end{const}


\begin{exam}
We have already seen some examples of the above construction.
\begin{enumerate}
\item For a braided monoidal category $\CB$, there is a canonical braided monoidal functor $\bar\CB \to Z(\CB)$ defined by $x\mapsto (x, \bar c_{x,-})$. The enriched monoidal category $\CB^\sharp$ constructed in Example \ref{exam:B1} and Example \ref{exam:B2} coincides with the canonical construction from the pair $(\CB,\CB)$.
\item The enriched monoidal category $Z(\CB^\sharp)$ in Theorem \ref{thm:self-enrich} coincides with the canonical construction from the pair $(\CB', \CB)$, where $\CB'$ is the M\"{u}ger center of $\CB$, i.e. the full subcategory of $\CB$ consisting of all transparent objects.
\end{enumerate}
\end{exam}

Given a braided oplax monoidal functor $\psi:\CD\to\CE$, the {\em centralizer} of $\CD$ in $\CE$ is the fully subcategory of $\CE$ consisting of those objects whose double braidings with the essential image of $\psi$ are all trivial.

\begin{thm}
Let $\CC^\sharp$ be the enriched monoidal category from Construction \ref{const:enrich}. The underlying category of the Drinfeld center $Z(\CC^\sharp)$ is the centralizer of $\bar\CB$ in $Z(\CC)$. Moreover, $Z(\CC^\sharp)(x,y)$ for $x,y\in Z(\CC^\sharp)$ represents the functor $Z(\CC)(\psi(-)\otimes x,y): \CB^\op\to Set$.
\end{thm}

\begin{proof}
Let $b_x$ be a half-braiding for an object $x\in\CC^\sharp$. Since $b_x$ is an enriched natural isomorphism, we have a commutative diagram for $y,z\in\CC^\sharp$:
$$
\xymatrix@!C=25ex{
  [y,z] \ar[r]^-{\otimes\circ(1\otimes\id_x)} \ar[d]_{\otimes\circ(\id_x\otimes 1)} & [y\otimes x,z\otimes x] \ar[d]^{-\circ b_{x,y}} \\
  [x\otimes y,x\otimes z] \ar[r]^-{b_{x,z}\circ-} & [x\otimes y,z\otimes x]. \\
}
$$
This diagram is equivalent to the following one via the adjunction between $\phi(-)\otimes x\otimes y$ and $[x\otimes y,-]$
\begin{equation} \label{eqn:brd}
\raisebox{2em}{\xymatrix@!C=15ex{
  \phi[y,z]\otimes x\otimes y \ar[rr]^-{1 \otimes b_{x,y}} \ar[d]_{b_{\phi[y,z],x}\otimes 1} & & \phi[y,z]\otimes y\otimes x \ar[d]^{v_z\otimes 1} \\
   x\otimes\phi[y,z]\otimes y \ar[r]^-{1\otimes v_z} & x\otimes z \ar[r]^{b_{x,z}} & z\otimes x. \\
}}
\end{equation}
Taking $y=\one$, we obtain a commutative diagram:
\begin{equation*} \label{eqn:brd2}
\xymatrix@!C=18ex{
  & \phi[\one,z]\otimes x \ar[ld]_{b_{\phi[\one,z],x}} \ar[rd]^{v_z\otimes 1} \\
   x\otimes\phi[\one,z] \ar[r]^-{b_{x,\phi[\one,z]}} & \phi[\one,z]\otimes x \ar[r]^{v_z\otimes 1} & z\otimes x. \\
}
\end{equation*}
Let $z=\phi(w)$ where $w\in\bar\CB$. Note that the functor $[\one,-]:\CC\to\CB$ is right adjoint to $\phi$. So, the composition $z\xrightarrow{\phi(u_w)}\phi[\one,z]\xrightarrow{v_z}z$ is the identity morphism. We obtain the following commutative diagram:
$$
\xymatrix{z \otimes x  \ar[rr]^-{\phi(u_w)\otimes 1} \ar[d]_{b_{z,x}} & & \phi[\one,z] \otimes x \ar[dr]^{v_z\otimes 1} \ar[d]_{b_{\phi[\one,z],x}}  & \\
x\otimes z \ar[rr]^-{1\otimes\phi(u_w)} \ar[d]_{b_{x,z}} & & x \otimes \phi[\one, z] \ar[d]_{b_{x,\phi[\one,z]}} & z\otimes x, \\
z\otimes x \ar[rr]^-{\phi(u_w)\otimes 1} & & \phi[\one,z] \otimes x \ar[ur]_{v_z\otimes 1}  &
}
$$
where the commutativity of the upper square follows from the fact that $\phi(u_w)$ is a morphism in $Z(\CC)$ (preserving half-braiding); that of the lower square follows from the naturality of the half-braiding $b_{x,-}$. Then we obtain immediately $b_{x,z}\circ b_{z,x}=1$ for all $z=\phi(w), w\in \bar\CB$. This shows that $(x,b_x)$ lies in the centralizer of $\bar\CB$ in $Z(\CC)$.

Conversely, if $(x,b_x)\in Z(\CC)$ lies in the centralizer of $\bar\CB$, then we have $b_{\phi[y,z],x}=b_{x,\phi[y,z]}^{-1}$. It is easy to see that Diagram \eqref{eqn:brd} commutes for all $y,z\in\CC^\sharp$ if replacing the left vertical morphism $b_{\phi[y,z],x}$ by $b_{x,\phi[y,z]}^{-1}$. Hence, $(x,b_x)\in Z(\CC^\sharp)$. This proves the first part of the theorem.

By definition, $Z(\CC^\sharp)((x,b_x),(y,b_y))$ is the maximal subobject $\iota:t\hookrightarrow [x,y]$ rendering the following diagram commutative for all $z\in\CC^\sharp$:
$$
\xymatrix@!C=18ex{
  t \ar[r]^-{\otimes\circ(\id_z\otimes\iota)} \ar[d]_{\otimes\circ(\iota\otimes\id_z)} & [z\otimes x,z\otimes y] \ar[d]^{-\circ b_{x,z}} \\
  [x\otimes z,y\otimes z] \ar[r]^-{b_{y,z}\circ-} & [x\otimes z,z\otimes y]. \\
}
$$
This diagram is equivalent to the following one via the adjunction between $\phi(-)\otimes x\otimes z$ and $[x\otimes z,-]$
$$
\xymatrix@!C=15ex{
  \phi(t)\otimes x\otimes z \ar[r]^-{\phi(\iota)\otimes 1} \ar[d]^{\phi(\iota)\otimes 1} & \phi[x,y]\otimes x\otimes z \ar[r]^{1 \otimes b_{x,z}} & \phi[x,y]\otimes z\otimes x \ar[r]^{b_{\phi[x,y],z}\otimes 1} & z\otimes\phi[x,y]\otimes x \ar[d]^{1 \otimes v_y} \\
  \phi[x,y]\otimes x\otimes z \ar[r]^-{v_y\otimes 1} & y\otimes z \ar[rr]^{b_{y,z}} & & z\otimes y. \\
}
$$
Therefore, $t$ is the maximal subobject such that the composite morphism $\phi(t)\otimes x \xrightarrow{\phi(\iota)\otimes 1} \phi[x,y]\otimes x \xrightarrow{v_y} y$ in $\CC$ preserves half-braiding, i.e. defines a morphism in $Z(\CC)$. In other words, $\CB(-,t) \simeq Z(\CC)(\psi(-)\otimes(x,b_x),(y,b_y))$.
This proves the second part of the theorem.
\end{proof}

A nonzero multi-fusion category is called {\em indecomposable} if it is not a direct sum of two nonzero multi-fusion categories.

\begin{cor} \label{cor:fusion}
Let $\CB$ be a nondegenerate braided fusion category, and let $\CC$ be an indecomposable multi-fusion category equipped with a $\C$-linear additive braided monoidal functor $\psi:\bar\CB\to Z(\CC)$.  Let $\CC^\sharp$ be the enriched monoidal category from Construction \ref{const:enrich}. We have $Z(\CC^\sharp) \simeq \bar\CB'$, where $\bar\CB'$ is the centralizer of $\bar\CB$ in $Z(\CC)$.
\end{cor}

\begin{proof}
Since both $\CB$ and $\CC$ are semisimple, $\CC$ has internal hom in $\CB$ thus $\CC^\sharp$ is well-defined.
By \cite[Corollary 3.26]{dmno}, the functor $\psi:\bar\CB\to Z(\CC)$ is fully faithful.
Then we have $Z(\CC)\simeq\bar\CB\boxtimes\bar\CB'$ by \cite[Theorem 3.13]{DGNO}. This implies that $Z(\CC^\sharp)(x,y)$ lies in the full subcategory of $\CB$ consisting of the direct sums of the tensor unit $\one$, which is nothing but the category of finite-dimensional vector spaces. Therefore, $Z(\CC^\sharp)(x,y) \simeq \bar\CB'(x,y)\one$, as desired.
\end{proof}



\end{document}